\documentclass[12pt,reqno]{amsart}
\usepackage{amsmath,amssymb,amsfonts,amscd,latexsym,amsthm,mathrsfs,verbatim,comment,color}
\usepackage{enumerate}
\usepackage{hyperref}
\newtheorem{theorem}{Theorem}
\newtheorem{lemma}{Lemma}

\theoremstyle{definition}
\newtheorem{remark}{Remark}
\newtheorem{exercise}{Exercise}
\renewcommand{\d}{{\mathrm d}}
\newcommand{\md}{\kern1.5pt|\kern1.5pt}
\let\<\langle
\let\>\rangle
\renewcommand{\Im}{\operatorname{Im}}
\newcommand{\ord}{\operatorname{ord}}
\newcommand{\Q}{\mathbb{Q}}
\newcommand{\Z}{\mathbb{Z}}
\newcommand{\cK}{\mathcal{K}}
\newcommand{\cT}{\mathcal{T}}
\begin{document}

\title{Magnetic (quasi-)modular forms}

\author{Vicen\c tiu Pa\c sol}
\address{Simion Stoilow Institute of Mathematics of the Romanian Academy, P.O.\ Box 1-764, 014700 Bucharest, Romania}
\email{vicentiu.pasol@imar.ro}

\author{Wadim Zudilin}
\address{Department of Mathematics, IMAPP, Radboud University, PO Box 9010, 6500~GL Nijmegen, Netherlands}
\email{w.zudilin@math.ru.nl}

\date{8 October 2020. \emph{Revised}: 3 February 2022}



\begin{abstract}
A (folklore?) conjecture states that no \emph{holomorphic} modular form $F(\tau)=\sum_{n=1}^\infty a_nq^n\in q\mathbb Z[[q]]$ exists, where $q=e^{2\pi i\tau}$, such that its anti-derivative $\sum_{n=1}^\infty a_nq^n/n$ has integral coefficients in the $q$-expansion.
A recent observation of Broadhurst and Zudilin, rigorously accomplished by Li and Neururer, led to examples of \emph{meromorphic} modular forms possessing the integrality property.
In this note we investigate the arithmetic phenomenon from a systematic perspective and discuss related transcendental extensions of the differentially closed ring of quasi-modular forms.
\end{abstract}

\maketitle

\section{Introduction}
\label{intro}

One of the arithmetic features of modular and quasi-modular forms is integrality of the coefficients in their Fourier expansions. This is trivially seen on the generators
\begin{equation}
E_2(\tau)=1-24\sum_{n=1}^\infty\frac{nq^n}{1-q^n},
\quad
E_4(\tau)=1+240\sum_{n=1}^\infty\frac{n^3q^n}{1-q^n},
\quad
E_6(\tau)=1-504\sum_{n=1}^\infty\frac{n^5q^n}{1-q^n}
\label{eis-ser}
\end{equation}
of the ring of quasi-modular forms, as well as on the `discriminant' cusp form
$$
\Delta(\tau)=q\prod_{m=1}^\infty(1-q^m)^{24}=\frac{E_4^3-E_6^2}{1728},
$$
where $q=q(\tau)=e^{2\pi i\tau}$ for $\tau$ from the upper half-plane $\Im\tau>0$.
All $q$-expansions above converge for $q$ inside the unit disk, and in fact have polynomial growth of the coefficients.
A more suprising fact, brought to the mathematical community by Ramanujan \cite{Ra16} more than 100 years ago, is that the three Eisenstein series in \eqref{eis-ser} satisfy the algebraic system of differential equations
\begin{equation}
\delta E_2=\frac1{12}(E_2^2-E_4), \quad \delta E_4=\frac13(E_2E_4-E_6), \quad \delta E_6=\frac12(E_2E_6-E_4^2),
\label{rama-DE}
\end{equation}
where
$$
\delta=\frac1{2\pi i}\frac{\d}{\d\tau}=q\frac{\d}{\d q}.
$$
Ramanujan's notation for the Eisenstein series \eqref{eis-ser} was $P(q),Q(q),R(q)$, respectively, as he mainly viewed them as functions of the $q$-nome.
Since the functions $E_2,E_4,E_6$ are algebraically independent over $\mathbb C$, and even over $\mathbb C(q)$ and over $\mathbb C(\tau,q)$ \cite{Ma69,Re66}, this fine structure gives rise to remarkable applications in transcendental number theory to the values of quasi-modular forms.
One particular notable example in this direction is a famous theorem of Nesterenko \cite{Ne96}, which states that, given a complex number $q$ with $0<|q|<1$, at least three of the four quantities $q,P(q),Q(q),R(q)$ are algebraically independent over~$\mathbb Q$.

Establishing transformation properties of a double integral, which characterises the output voltage of a Hall plate affected by the shape of the plates and sizes of the contacts and which is\,---\,for this reason\,---\,dubbed \emph{magnetic}, in the work \cite{BZ19} Broadhurst and the second author came across a \emph{meromorphic} modular form (on a congruence subgroup), whose anti-derivative had integral coefficients in its $q$-expansion and was not a modular object itself.
This arithmetic observation was subsequently proven by Li and Neururer in \cite{LN19} who also noticed
that the formal anti-derivative
$$
\tilde F_{4a}=\delta^{-1}\biggl(\frac{\Delta}{E_4^2}\biggr)=\int_0^q\frac{\Delta}{E_4^2}\,\frac{\d q}q
$$
of the meromorphic modular form $F_{4a}(\tau)=\Delta/E_4^2$ has integer coefficients in its $q$-expansion. (They proved a slightly weaker version about the integrality of the anti-derivative of $64\Delta/E_4^2$.) The function $F_{4a}(\tau)$ has weight~4 and possesses the double pole at $\tau=\rho=e^{2\pi i/3}$ in the fundamental domain, and a simple analysis reveals that it is not the image under $\delta$ of an element from the (differentially closed) field $\mathbb C(q,E_2,E_4,E_6)$. This implies that the anti-derivative $\tilde F_{4a}=\delta^{-1}F_{4a}$ is transcendental over the field, hence the addition of $\tilde F_{4a}$ to the latter increases the transcendence degree by~1.
Following the background in \cite{BZ19}, Li and Neururer coined the name `magnetic modular form' to a meromorphic modular form like $F_{4a}$.
A principal goal of this note is to investigate the `magnetic modular' phenomenon further and to give more examples of those.

\begin{theorem}
\label{th1}
The meromorphic modular forms $F_{4a}(\tau)=\Delta/E_4^2$ and $F_{4b}(\tau)=E_4\Delta/E_6^2$ of weight $4$ are magnetic.
In other words, their anti-derivatives $\delta^{-1}F_{4a}$ and $\delta^{-1}F_{4b}$ have integral $q$-expansions.
\end{theorem}

\begin{theorem}
\label{th2}
The meromorphic modular form $F_6(\tau)=E_6\Delta/E_4^3$ of weight $6$ is doubly magnetic\textup:
its first and second anti-derivatives $\delta^{-1}F_6$ and $\delta^{-2}F_6$ have integral $q$-ex\-pan\-sions.
\end{theorem}

There are other instances in the literature of related integrality phenomena;
however the existing methods of proofs seem to be quite different from what we use below.
Investigating the solution space of the linear differential equation
\begin{equation*}
D_kf(\tau)=0, \quad\text{where}\;
D_k=\delta^2-\frac{k+1}6E_2(\tau)\delta+\frac{k(k+1)}{12}\delta E_2(\tau),
\end{equation*}
in \cite{HK12} Honda and Kaneko found that, when $k=4$, it is spanned by $E_4$ and
$$
\tilde E_4=E_4\cdot\delta^{-1}\biggl(\frac{\Delta^{5/6}}{E_4^2}\biggr)\in q^{5/6}\mathbb Q[[q]].
$$
They numerically observed and proved some related results about the $p$-integrality of $\tilde E_4$ for primes $p\equiv1\bmod3$. This theme was later analysed and generalised in \cite{AA14,Gu13,Gu20}.
Bringing some parallel to that investigations, it is easy to check that the functions $E_4$ and $E_4\,\delta^{-1}(\Delta/E_4^2)$ (both with integer coefficients in their $q$-expansions!) span the solution space of the differential equation $Df=0$, where
$$
D=\delta^2-E_2\delta+\frac1{36}\biggl(7E_2^2-5E_4-2\frac{E_2E_6}{E_4}\biggr)
=D_5+\frac16\biggl(E_2\frac{\delta E_4}{E_4}-5\delta E_2\biggr).
$$
At the same time, the only quasi-modular solutions of $D_5y=0$ are spanned by $\delta E_4$ (see \cite[Theorem 2]{KK03}).

A somewhat different account of strong divisibility of the coefficients of modular forms shows up in the context of arithmetic properties of traces of singular moduli initiated in Zagier's work \cite{Za02}.
As this topic remains quite popular, we only list a selection of contributions \cite{Ah12,AO05,DJ08,Ed05,Gu07,Je05}.
The methods involved make use of the Shimura correspondence, which is also the main ingredient of our proof of Theorems~\ref{th1} and~\ref{th2}.

\section{Magnetic quasi-modular forms}
\label{V4}

In this part we formalise the notion of magnetic forms and give results, which may be thought of as generalisations of Theorems~\ref{th1} and \ref{th2} but use the theorems as principal steps.

Consider the family
$$
f_{a,b,c}=E_2^aE_4^bE_6^c, \quad\text{where}\; a,b,c\in\mathbb Z, \; a\ge0,
$$
of meromorphic quasi-modular forms. Their $q$-expansions all belong to $\mathbb Z[[q]]$.
For $k\in\mathbb Z$ even, denote by $W_k$ the $\mathbb{Q}$-vector space in $\mathbb Q\otimes_{\mathbb{Z}}\mathbb Z[[q]]$ (the $q$-series $f\in\mathbb Q[[q]]$ with $Nf\in\mathbb Z[[q]]$ for some $N\in\mathbb Z_{>0}$) spanned by the $q$-expansions of the forms $f_{a,b,c}$ of weight $k$, that is, with $2a+4b+6c=k$.
Because
\begin{equation}
\delta f_{a,b,c}=\frac{k-a}{12}f_{a+1,b,c}-\frac{a}{12}f_{a-1,b+1,c}-\frac{b}{3} f_{a,b-1,c+1}-\frac{c}{2} f_{a,b+2,c-1},
\label{delta}
\end{equation}
the differential operator $\delta$ defines a well defined map $W_k\to W_{k+2}$.
Clearly, the image $\delta W_k$ in $W_{k+2}$ is a $\mathbb Q$-subspace in $\mathbb Q\otimes_{\mathbb{Z}}q\Z[[q]]$; we will call $W_{k+2}^0$ the cuspidal subspace of $W_{k+2}$, that is, the set of all elements in $W_{k+2}$ with vanishing constant term in their $q$-expansion.

We will say that an element $v\in W_k^0$ is \emph{magnetic} if its formal anti-derivative
$$
\delta^{-1}v=\int_0^q v\,\frac{\d q}{q}\in \mathbb{Q}\otimes_{\mathbb{Z}} q\mathbb{Z}[[q]].
$$
We also call it \emph{strongly magnetic} if $\delta^{-1}v\in q\mathbb{Z}[[q]]$.
With the magnetic property, we can associate the equivalence relation $\sim$ on $W_k$ writing $v\sim w$ if and only if the difference $v-w$ is in $W_k^0$ and is magnetic.

Let $V_k$ (respectively, $V_k^0$) be the $\Q$-vector subspace of $W_k$ (respectively, of $W_k^0$) generated by the forms $f_{a,b,c}$ with $a\in\{0,1,\dots,k-2\}$.
According to relation \eqref{delta} this range of $a$ makes the subspace $V_k$ \emph{stable} under the $\delta$-differentiation.
Notice that $\delta V_2\subseteq V_{4}^0$. 

\begin{theorem}
\label{th:w4}
Any element of $V_4^0$ is magnetic.
\end{theorem}

\begin{remark}
\label{rem:w4}
It seems that the elements of $W_4^0$ with $a>2$ (that is, outside the range assumed in $V_4^0$) with the magnetic property are those that come as linear combinations of $\delta$-derivatives of elements from~$W_2$.
In other words, we expect that the choice of $V_4^0$ in the theorem as a magnetic space of weight~4 to be sharp.
\end{remark}

\begin{proof}[Derivation of Theorem~\textup{\ref{th:w4}} from Theorem~\textup{\ref{th1}}]
It follows from Theorem~\textup{\ref{th1}} that the forms
$$
f_{0,1,0}-f_{0,-2,2}=1728F_{4a}
\quad\text{and}\quad
f_{1,2,-1}-f_{0,1,0}
=6\delta f_{0,2,-1}-5184F_{4b}
$$
are magnetic; in other words, we have the equivalences $f_{0,-2,2}\sim f_{0,1,0}$ and $f_{1,2,-1}\sim f_{0,1,0}$.

Any element in $V_4$ can be written as $E_2^aP(E_4,E_6)/(E_4^mE_6^n)$, for some $a,m,n$ non-negative integers, $a\le2$, and $P(x,y)\in\mathbb Q[x,y]$. Such an expression clearly splits into a linear combination of the form $f_{a,b,c}\in V_4$ with $0\le a\le2$ and either $b\ge0$ or $c\ge0$.
If both $b\ge0$ and $c\ge0$ then we get only two elements in $V_4$, namely, $f_{0,1,0}$ and $f_{2,0,0} = f_{0,1,0} + 12\delta f_{1,0,0}$, both equivalent to $f_{0,1,0}$. Therefore, we only need to prove the theorem in two situations:
$b\ge0$ and $c<0$, or $b<0$ and $c\ge0$.

If $b\ge0$ and $c<0$, then there is only one form $f_{a,b,c}\in V_4$ with $c=-1$.
Indeed, solving $4=2a+4b+6c=2a+4b-6$ we get $a=1$, $b=2$. By the hypothesis, this form $f_{1,2,-1}\sim f_{0,1,0}$.
For $c\le-2$ we use equation \eqref{delta} (with $k=2$) in the form
\begin{equation*}
\frac{c+1}{2}f_{a,b,c}=-\delta f_{a,b-2,c+1}
-\frac{a}{12}f_{a-1,b-1,c+1}-\frac{b-2}{3} f_{a,b-3,c+2}-\frac{a-2}{12} f_{a+1,b-2,c+1},
\end{equation*}
and induction on $-c$ to conclude that $f_{a,b,c}$ is equivalent to a linear combination of $f_{1,2,-1}$ and $f_{0,1,0}$, hence to $f_{0,1,0}$ alone. (Notice that prefactors $a/12$ and $(a-2)/12$ leave the terms on the right-hand side in~$V_4$.)

If $b<0$ (and $c\ge0$), we use equation \eqref{delta} in the form
\begin{equation}
\frac{b+1}{3} f_{a,b,c}
=-\delta f_{a,b+1,c-1}
-\frac{a-2}{12}f_{a+1,b+1,c-1}
-\frac{a}{12}f_{a-1,b+2,c-1}
-\frac{c-1}{2} f_{a,b+3,c-2}.
\label{delta2}
\end{equation}
When $b=-1$ and $b=-2$, the only forms $f_{a,b,c}\in V_4$ possible with $c\ge0$ are $f_{1,-1,1}$ and $f_{0,-2,2}$, respectively. Substituting $a=0$, $b=-2$, $c=2$ in \eqref{delta2} leads to
$$
-\frac{1}{3} f_{0,-2,2}
=-\delta f_{0,-1,1}
+\frac{1}{6}f_{1,-1,1}
-\frac{1}{2} f_{0,1,0}
$$
implying $f_{1,-1,1}\sim f_{0,-2,2}\sim f_{0,1,0}$ from the hypothesis.
For $b\le-3$ we use \eqref{delta2} to conclude by induction on $-b$ that any such $f_{a,b,c}$ is equivalent to a linear combination of $f_{0,-2,2}$, $f_{1,-1,1}$ and $f_{0,1,0}$, hence to $f_{0,1,0}$.
This completes the proof of the theorem.
\end{proof}

\begin{remark}
\label{rem2}
It follows from the proof that we can replace the generator $f_{0,-2,2}-f_{0,1,0}$ with $f_{1,-1,1}-f_{0,1,0}$.
Furthermore, alternative choices for $f_{0,-2,2}-f_{0,1,0}$ and $f_{1,2,-1}-f_{0,1,0}$ are $\tilde F_j=E_2\cdot(\delta E_j)/E_j$ or $\hat F_j=(\delta^2E_j)/E_j$ for $j=4,6$.
\end{remark}

For weight $6$ the situation is slightly different. Only the following is true.

\begin{theorem}
\label{th:w6}
Let $U_6$ be the subspace of $V_6$ spanned over $\mathbb Q$ by $f_{a,b,c}$ with the additional constraint $c\ge 0$,
and $U_6^0=U_6\cap V_6^0$ its cuspidal subspace.
Then any element of $U_6^0$ is magnetic.
\end{theorem}

\begin{remark}
\label{rem3}
In fact, it seems that the space $U_6^0$ possesses the \emph{strongly} magnetic property: the anti-derivative of any difference of two $f_{a,b,c}$ from $U_6$ has an integral $q$-expansion.
\end{remark}

\begin{proof}
For $c=0$, we only have two elements $f_{3,0,0}$ and $f_{1,1,0}$ in $U_6$,
and $f_{3,0,0}\sim f_{1,1,0}$ since $ f_{3,0,0}- f_{1,1,0}=6\delta f_{2,0,0}$.
Moreover, they are both strongly equivalent to $f_{0,0,1}$, because $f_{1,1,0}-f_{0,0,1}=3\delta E_4$.

For $c=1$, we find out that $f_{0,0,1}$, $f_{2,-1,1}$ and $f_{4,-2,1}$ are in $U_6$.
Then $f_{4,-2,1}$ is strongly equivalent to any of $f_{3,0,0}$, $f_{1,1,0}$ and $f_{0,0,1}$ in accordance with $ f_{4,-2,1}-f_{3,0,0}=3\delta f_{4,-1,0}$ and the above.
With the help of Theorem~\ref{th2} and derivation
$$
f_{2,-1,1}-f_{0,0,1}
=4\delta f_{1,-1,1}-4\delta f_{0,-2,2}+2(f_{1,1,0}-f_{0,0,1})-4608F_6,
$$
we see that the same is true for $f_{2,-1,1}$.

We have just shown that any element in the subspace $U_6^0$ generated by $f_{a,b,c}$ with $c\in\{0,1\}$ does have the (strongly) magnetic property.
For the rest of our theorem, we proceed by induction over $c$ using the following consequence of equation~\eqref{delta} when $k=4$:
\begin{equation*}
\frac{b}{3} f_{a,b-1,c+1}=-\delta f_{a,b,c}+\frac{4-a}{12}f_{a+1,b,c}-\frac{a}{12}f_{a-1,b+1,c}-\frac{c}{2} f_{a,b+2,c-1}.
\qedhere
\end{equation*}
\end{proof}

\section{A magnetic extension of the field of quasi-modular forms}
\label{non-surj}

The functions $\tau,q,E_2,E_4,E_6$ are algebraically independent over $\mathbb C$ (see \cite{Ma69,Re66}).
We can identify the differential field $\mathbb C\<\tau,q,E_2,E_4,E_6\>$ generated by them over $\mathbb C$ with the differential field $\cK=\mathbb C\<\tau,q,X,Y,Z\>$ equipped with the derivation
$$
D=\frac1{2\pi i}\,\frac\partial{\partial\tau}+q\frac\partial{\partial q}
+\frac1{12}(X^2-Y)\frac\partial{\partial X}+\frac13(XY-Z)\frac\partial{\partial Y}+\frac12(XZ-Y^2)\frac\partial{\partial Z}.
$$
Our goal is to demonstrate that the elements
$$
v_1=\frac{XZ}{Y}-Y \quad\text{and}\quad v_2=\frac{XY^2}{Z}-Z,
$$
corresponding to $f_{1,-1,1}-f_{0,1,0}$ and $f_{1,2,-1}-f_{0,1,0}$, do not have $D$-anti-derivatives in $\cK$
(not even in $\cK\<D^{-1}v_2\>$ and $\cK\<D^{-1}v_1\>$, respectively).
This follows trivially from noticing that $\ord_Yv_1=-1$ and $\ord_Zv_2=-1$, so that if either $D^{-1}v_1$ or $D^{-1}v_2$ existed then $\ord_YD^{-1}v_1<0$ and $\ord_ZD^{-1}v_2<0$, hence $\ord_Yv_1=\ord_YD(D^{-1}v_1)\le-2$ and similarly $\ord_Zv_2\le-2$, contradiction.

By \cite[Lemma 3.9]{Ka57} applied twice, the anti-derivatives
$$
\tilde E_{4a}=\delta^{-1}(f_{1,-1,1}-f_{0,1,0})
\quad\text{and}\quad
\tilde E_{4b}=\delta^{-1}(f_{1,2,-1}-f_{0,1,0})
$$
are algebraically independent over the field $\mathbb C\<\tau,q,E_2,E_4,E_6\>$, the extended differential field
$$
\mathbb C\<\tau,q,E_2,E_4,E_6,\tilde E_{4a},\tilde E_{4b}\>
$$
has transcendence degree 7 over $\mathbb C$ and is a Picard--Vessiot extension of the differential field $\mathbb C\<\tau,q,E_2,E_4,E_6\>$. Again, by identifying the former through the isomorphism
$$
\varphi\colon E_2\mapsto X, \; E_4\mapsto Y, \; E_6\mapsto Z, \; \tilde E_{4a}\mapsto S, \; \tilde E_{4b}\mapsto T
$$
with the differential field $\hat{\cK}=\mathbb C\<\tau,q,X,Y,Z,S,T\>$ equipped with the derivation
\begin{align*}
\hat D&=\frac1{2\pi i}\,\frac\partial{\partial\tau}+q\frac\partial{\partial q}
+\frac1{12}(X^2-Y)\frac\partial{\partial X}+\frac13(XY-Z)\frac\partial{\partial Y}+\frac12(XZ-Y^2)\frac\partial{\partial Z}
\\ &\qquad
+\biggl(\frac{XZ}{Y}-Y\biggr)\frac\partial{\partial S}
+\biggl(\frac{XY^2}{Z}-Z\biggr)\frac\partial{\partial T},
\end{align*}
we want to demonstrate that the element
$$
v_3=\frac{X^2Z}{Y}-Z
$$
corresponding to $f_{2,-1,1}-f_{0,0,1}$ does not have a $\hat D$-anti-derivative in $\hat{\cK}$.

Assume on the contrary that there is an element $u_3\in\hat{\cK}$ such that $\hat Du_3=v_3$.
Notice that the functions $\tau$, $q=e^{2\pi i\tau}$, $E_2(\tau)$, $E_4(\tau)$ and $E_6(\tau)$ are all analytic at $\tau=\rho=e^{2\pi i/3}$, the latter three having the values
$$
E_2(\rho)=\frac{2\sqrt3}{\pi}, \quad E_4(\rho)=0, \quad E_6(\rho)=\biggl(\frac{3\Gamma(\frac13)^6}{8\pi^4}\biggr)^3.
$$
With the help of Ramanujan's system \eqref{rama-DE} we find out that
$$
E_4(\tau)=-\frac{2\pi i}3\,E_6(\rho)(\tau-\rho)+O\bigl((\tau-\rho)^2\bigr) \quad\text{as}\; \tau\to\rho,
$$
so that
$$
\begin{aligned}
f_{1,-1,1}-f_{0,1,0}
&=\frac{3iE_2(\rho)}{2\pi}\,\frac1{\tau-\rho}+O(1),
\\
f_{2,-1,1}-f_{0,0,1}
&=\frac{3iE_2(\rho)^2}{2\pi}\,\frac1{\tau-\rho}+O(1) 
\end{aligned}
\quad\text{as}\; \tau\to\rho
$$
and $f_{1,2,-1}-f_{0,1,0}$ is analytic at $\tau=\rho$. In turn, this implies that
$$
\begin{aligned}
\tilde E_{4a}&=-3E_2(\rho)\ln(\tau-\rho)+g_1(\tau),
\\
\delta^{-1}(f_{2,-1,1}-f_{0,0,1})
&=-3E_2(\rho)^2\ln(\tau-\rho)+g_3(\tau)
\end{aligned}
\quad\text{as}\; \tau\to\rho
$$
for some functions $g_1(\tau)$ and $g_3(\tau)$ analytic at $\tau=\rho$, while $\tilde E_{4b}(\tau)$ is analytic there.
To summarise, the function
$$
\delta^{-1}(f_{2,-1,1}-f_{0,0,1})
-\frac{2\sqrt3}\pi\,\tilde E_{4a}(\tau)
=\delta^{-1}(f_{2,-1,1}-f_{0,0,1})
-E_2(\rho)\tilde E_{4a}(\tau)
$$
is analytic at $\tau=\rho$, hence only representable as a rational function of $\tau,q,E_2,E_4,\allowbreak E_6,\tilde E_{4b}$.
Using the isomorphism $\varphi$ we conclude that
$$
u=u_3-\frac{2\sqrt3}\pi\,S\in\hat{\cK}
$$
is a polynomial in $\tau,q,X,Y,Z,T$. The latter is seen to be impossible after the operator $\hat D$ is applied to $u$ and to $u_3-\dfrac{2\sqrt3}\pi\,S$ leading to a rational expression of $S$ in terms of the other generators of~$\hat{\cK}$.
The contradiction we arrive at implies that the anti-derivative
$$
\tilde E_6=\delta^{-1}(f_{2,-1,1}-f_{0,0,1})
$$
is transcendental over the field $\mathbb C\<\tau,q,E_2,E_4,E_6,\tilde E_{4a},\tilde E_{4b}\>$.
On replacing the generators of the latter with the anti-derivatives of magnetic modular forms from Theorems~\ref{th1} and~\ref{th2} we obtain the following result.

\begin{theorem}
\label{th555}
The differentially closed field
$$
\mathbb C\<\tau,q,E_2,E_4,E_6,\tilde F_{4a},\tilde F_{4b},\tilde F_6\>,
$$
generated by $\tau$, $q=e^{2\pi i\tau}$, the Eisenstein series \eqref{eis-ser} and the anti-derivatives
$$
\tilde F_{4a}=\delta^{-1}\biggl(\frac{\Delta}{E_4^2}\biggr), \quad
\tilde F_{4b}=\delta^{-1}\biggl(\frac{E_4\Delta}{E_6^2}\biggr), \quad
\tilde F_6=\delta^{-1}\biggl(\frac{E_6\Delta}{E_4^3}\biggr)
$$
with integral coefficients in their $q$-expansions,
has transcendence degree $8$ over $\mathbb C$.
\end{theorem}

\begin{remark}
\label{rem-alt}
Another way to see that no $u_3$ exists in $\hat{\cK}$ such that $\hat Du_3=v_3$ is by casting $u_3$ in the form $p/q$ with 
$p,q$ in the ring $\mathcal R[S]$, where $\mathcal R=\mathbb C\<\tau,q,X,Y,Z,T\>$, and $\gcd(p,q)=1$.
After clearing the denominators in $\hat D(p/q)=v_3$ and comparing the degree in $S$ on both sides, one concludes that $\hat Dq=uq$ for some $u\in\mathcal R$ (that is, independent of~$S$).
This leads to conclusion $q\in\mathcal R$, so that $u_3$ is a polynomial in~$S$.
Finally, the equation $\hat Du_3=X^2Z/Y-Z$ is seen to be impossible by comparing the order in~$Y$ on both sides.
\end{remark}

\begin{exercise}
\label{ex1}
We leave to the reader the exercise to prove that the anti-derivative of $\tilde F_6$ (in turn, the second anti-derivative of $F_6$) is transcendental over the field in Theorem~\ref{th555}.
\end{exercise}

\section{Half-integral weight weakly holomorphic modular forms}
\label{w5/2}

Following the ideas in \cite{LN19}, we will cast magnetic modular forms of weight $2k$ as the images of weakly holomorphic eigenforms of weight $k+1/2$ under the Shimura--Borcherds lift.
In our settings, an input for the lift is a form $f(\tau)=\sum_{n\gg-\infty}a(n)q^n$ from the Kohnen plus space $M_{k+1/2}^{!,+}$ (meaning that $a(n)$ vanishes when $(-1)^kn\not\equiv0,1\bmod4$);
the output is the meromorphic modular form $\Psi(f)(\tau)=\sum_{n>0}A(n)q^n$ with
\begin{equation}
A(n)=\sum_{d\mid n}\bigg(\frac dD\bigg)d^{k-1}a(|D|\,n^2/d^2),
\label{a->A}
\end{equation}
where $D=D_k=1$ for $k$ even (so that the Kronecker--Jacobi symbol $\big(\frac dD\big)$ is always~1) and $D=D_k=-3$ for $k$~odd.
In other words,
\begin{equation}
\Psi=\Psi_k\colon f=\sum_{n\gg-\infty}a(n)q^n
\mapsto F=\sum_{n>0}q^n\sum_{d\mid n}\bigg(\frac d{D_k}\bigg)d^{k-1}a(|D_k|\,n^2/d^2),
\label{SB}
\end{equation}
and the latter expression is just $F=\sum_{n>0}q^n\sum_{d\mid n}d^{k-1}a(n^2/d^2)$ when $k$~is even.
We will also distinguish the Kohnen plus cuspidal space $S_{k+1/2}^{!,+}$ in $M_{k+1/2}^{!,+}$ by imposing the additional constraint $a(0)=0$.

Our examples of forms from $M_{k+1/2}^{!,+}$ with $k=2$ involved in the proof of Theorem~\ref{th1} are the following three:
\begin{align*}
g_0(\tau)&=\theta(\tau)\,(\theta(\tau)^4-20 E_{2,4}(\tau))
\\
&=1 - 10q - 70q^4 - 48q^5 - 120q^8 - 250q^9 - \dotsb - 550q^{16} - \dotsb
\\ &\quad
- 1210q^{25} - \dotsb - 1750q^{36} - \dotsb - 3370q^{49} -\dotsb,
\displaybreak[2]\\
g_1(\tau)&=\frac{\theta(\tau)E_4(4\tau)^2E_6(4\tau)}{\Delta(4\tau)}
\\
&=q^{-4} + 2q^{-3} + 2 - 196884q^4 - \dotsb - 85975040q^9 - \dotsb
\\ &\quad
- 86169224844q^{16} - \dotsb - 51186246451200q^{25} - \dotsb
\\ &\quad
- 35015148280961780q^{36} - \dotsb - 21434928162930081792q^{49} - \dotsb,
\displaybreak[2]\\
g_2(\tau)&=\frac{g_0(\tau)E_4(4\tau)^3}{\Delta(4\tau)}
\\
&=q^{-4} - 10q^{-3} + 674 - 7488q + 144684q^4 - \dotsb - 224574272q^9 - \dotsb
\\ &\quad
- 42882054732q^{16} - \dotsb - 63793268216640q^{25} - \dotsb
\\ &\quad
- 31501841125150388q^{36} - \dotsb - 22385069000981561664q^{49} - \dotsb,
\end{align*}
where $\theta(\tau)=\sum_{n\in\mathbb Z}q^{n^2}$ and
\begin{equation}
E_{2,4}(\tau)
=\frac{-E_2(\tau)+3E_2(2\tau)-2E_2(4\tau)}{24}
=\sum_{\substack{n=1\\n\;\text{odd}}}^\infty q^n\sum_{d\mid n}d.
\label{E24}
\end{equation}
The modular form $g_0(\tau)$ is known by the name of normalised Cohen--Eisenstein series of weight~$5/2$.

\begin{lemma}
\label{lem1}
\begin{enumerate}[(a)]
\item[\textup{(a)}]
The weight $5/2$ weakly holomorphic modular form
$$
f_{4a}(\tau) = \frac{7}{8} g_0(\tau) + \frac{1}{768} g_1(\tau) - \frac{1}{768} g_2(\tau)
= \frac{1}{64}q^{-3} + q - 506q^4 + \dotsb
$$
lies in the Kohnen plus cuspidal space $S^{!,+}_{5/2}$ and its Shimura--Borcherds lift $\Psi(f_{4a})$ is $F_{4a}=\Delta/E_4^2$.
\item[\textup{(b)}]
The weight $5/2$ weakly holomorphic modular form
$$
f_{4b}(\tau) = \frac{19}{18} g_0(\tau) -\frac{5}{648} g_1(\tau) - \frac{1}{648} g_2(\tau)
= -\frac{1}{108}q^{-4} + q + 1222q^4 + \dotsb
$$
lies in the Kohnen plus cuspidal space $S^{!,+}_{5/2}$ and its Shimura--Borcherds lift $\Psi(f_{4b})$ is $E_4\Delta/E_6^2$.
\end{enumerate}

\noindent
Moreover, $f_{4a}\in \frac{1}{64} q^{-3}\Z[[q]]$ and $f_{4b}\in \frac{1}{108} q^{-4}\Z[[q]]$.
\end{lemma}

The identification $\Psi(f_{4a})=F_{4a}$ is already in Borcherds' \cite[Example~14.4]{Bo98}.

\begin{proof}
Indeed, we only need to check that $f_{4a},f_{4b}$ have vanishing constant term and that the first three coefficients in the $q$-expansions of $\Psi(f_{4a})$, $\Psi(f_{4b})$ agree with those of the predicted meromorphic modular forms; we choose to check the first seven coefficients.

For the integrality statement, we use the alternative expressions
$$
64f_{4a}(\tau)= \frac{f_{14+1/2}^*(\tau)}{\Delta(4\tau)}
$$
and 
$$
-108f_{4b}(\tau)= \frac{ f_{14+1/2}(\tau)\,(j(4\tau)-674)+10 f_{14+1/2}^*(\tau)}{\Delta(4\tau)},
$$
where the forms $f_{b+1/2}(\tau), f^*_{b+1/2}(\tau)$ are the holomorphic modular forms of weight $b+1/2$ with integral $q$-expansions from the table in \cite[Appendix]{DJ08} and $j(\tau)=E_4(\tau)^3/\Delta(\tau)$ is the elliptic modular invariant.
\end{proof}

As we will see further, for certain forms $\sum_{n\gg-\infty}a(n)q^n\in S^{!,+}_{5/2}$ with integral $q$-expansions (in particular, for the forms $64f_{4a}$ and $108f_{4b}$) one can make use of Hecke operators to conclude with the divisibility $n\mid a(n^2)$ for $n>0$.
This readily implies that $64F_{4a}$ and $108F_{4b}$ in Theorem~\ref{th1} are strongly magnetic modular forms, since the relation in~\eqref{a->A} translates the divisibility into
$$
\frac{A(n)}n
=\sum_{d\mid n}\frac{a(n^2/d^2)}{n/d}
=\sum_{d\mid n}\frac{a(d^2)}{d}\in\mathbb Z.
$$
A detailed analysis below reveals that the factors $64$ and $108$ can be also removed.

\section{The square part and Hecke operators}
\label{sq}

We refer the reader to \cite{DJ08} and \cite{BGK15} for the definition of Hecke operators $\cT_p$ and $T_{p^2}$ on integral weight $2k$ and half-integral weight $k+1/2$ modular forms (including weakly holomorphic or meromorphic), respectively.
As in the case of the Shimura--Borcherds lift $\Psi=\Psi_k$ in \eqref{SB}, these definitions make perfect sense for \emph{any} Laurent series $f=\sum_{n\gg-\infty}a(n)q^n$, not necessarily of modular origin but with the weight $2k$ or $k+1/2$ additionally supplied.
We refer to the finite sum $\sum_{n<0}a(n)q^n$ as to the principal part of~$f$.
We take
$$
f\md U_p=\sum_{n\gg-\infty}a(np)q^n,
\quad
f\md V_p=\sum_{n\gg-\infty}a(n)q^{np},
\quad
f\md\chi=\sum_{n\gg-\infty}\chi(n)a(n)q^n
$$
for a character $\chi\colon\Z\to\mathbb  C$, and define
$$
f\md \cT_p=f\md(\cT_p,2k)=f\md U_p+p^{2k-1} V_p
$$
and
$$
f\md T_{p^2}=f\md(T_{p^2},k+1/2)=f\md U_p^2+p^{k-1}\chi_p+p^{2k-1} V_p^2,
$$
where $\chi_p(n)=\chi_{p,k}(n)=\big(\frac{(-1)^kn}p\big)$ is the Kronecker--Jacobi symbol.

A simple calculation shows that $\Psi_k(f)\md(\cT_p,2k)=\Psi_k\big(f\md(T_{p^2},k+1/2)\big)$, which we can reproduce in a simplified form
\begin{equation*}
\Psi(f)\md\cT_p=\Psi(f\md T_{p^2})
\end{equation*}
when $k$ is fixed.

\begin{lemma}\label{Cong1}
Given a positive integer $k$, assume that there are no cusp forms of weight $2k$.
For a prime $p$, let $f\in M^{!,+}_{k+1/2}$ have $p$-integral coefficients and satisfy $p^2>-\ord_q(f)$.
Then
$$
f\md T_{p^2}^n\equiv 0 \bmod p^{(k-1)n}.
$$
\end{lemma}

\begin{proof}
Following the argument in \cite[proof of Lemma~3.1]{BGK15}, we can write
\begin{equation}
T_{p^2}^n
=\sum_{\substack{a,b,c,r\ge0\\a+b+c=n\\r\le\min\{a,c\}}}
\alpha_{a,b,c,r}\cdot p^{(2k-1)c+(k-1)b}\cdot U_{p^2}^{a-r}\chi_p^bV_{p^2}^{c-r},
\label{BGK}
\end{equation}
where $\alpha_{a,b,c,r}$ are some integers.
This writing can be easily deduced from $V_{p^2}\chi_p =\chi_p U_{p^2}=0$ and the fact that $V_{p^2}U_{p^2}$ is the identity.
We only need to analyse the principal part of $f\md T_{p^2}^n$ which, by the hypothesis $\dim S_{2k}=0$, determines it uniquely.

If $r<a$, then $f\md U_{p^2}^{a-r}\chi_p^b V_{p^2}^{c-r}$ has no principal part, because the latter is killed by a single  action of $U_{p^2}$ (since $a_{-p^2m}=0$ for any $m\ge 0$).
Therefore, we may assume that $a=r\le c$. This implies that $(2k-1)c+(k-1)b\ge (k-1)(2c+b)\ge(k-1)n$, hence the principal of $ f\md T_{p^2}^n$ part is divisible by $p^{(k-1)n}$.
This in turn implies that $f\md T_{p^2}^n=p^{(k-1)n}\cdot g$ for some $g\in \mathcal{M}_{k+1/2}^{!,+}$ with $p$-integral coefficients, since there is a basis $\{g_m=q^m+O(q):m\in\Z,\;(-1)^km\equiv0\}$ of $M^{!,+}_{k+1/2}$ whose elements have all coefficients integral (see \cite[Proposition~2]{DJ08}).
\end{proof}

In parallel with \eqref{SB}, define
\begin{equation*}
\Phi=\Phi_k\colon g=\sum_{n\gg-\infty}b(n)q^n
\mapsto \sum_{n>0}q^{|D_k|n^2}\sum_{d\mid n}\bigg(\frac d{D_k}\bigg)d^{k-1}\mu(d)b(n/d),
\end{equation*}
where $\mu(\,\cdot\,)$ is the M\"obius function and, as before, $D_k=2\cdot(-1)^k-1\in\{1,-3\}$.
We further define the `square part' of a Laurent series $f=\sum_{n\gg-\infty}a(n)q^n$ as
$$
f^{\square}=\sum_{n>0} a(|D_k|n^2) q^{|D_k| n^2}.
$$
The definitions immediately lead to the following conclusions.

\begin{lemma}\label{sqr}
We have $\Phi (\Psi(f))=f^{\square}$.
In particular, if $\Psi(f)\in q\Z[[q]]$, then $f^{\square}\in q\Z[[q]]$.
\end{lemma}

Notice that $f_{4a}^\square, f_{4b}^\square \in q\Z[[q]]$ by this lemma, because both $F_{4a}=\Psi(f_{4a})$ and $F_{4b}=\Psi(f_{4b})$ are in $q\Z[[q]]$.

In addition to this, we list some other easily verifiable properties about the interaction of Hecke operators and square parts.

\begin{lemma}\label{Heck}
Given a Laurent series $f=\sum_{n\gg-\infty}a(n)q^n$ and positive integer $k$, the following statements are true.
\begin{enumerate}
\item[\textup{(a)}] $\Psi(f)\md\cT_p^n=\Psi(f\md T_{p^2}^n)$ for $n=1,2,\dots$\,.
\item[\textup{(b)}] $\Psi(f)=\Psi(f^\square)$.
\item[\textup{(c)}] 
$(f\md T_{p^2})^{\square}=f^{\square}\md T_{p^2}$ termwise, that is,
$(f\md U_{p^2})^{\square}=f^{\square}\md U_{p^2}$,
$(f\md V_{p^2})^{\square}=f^{\square}\md V_{p^2}$ and
$(f\md \chi_p)^{\square}=f^{\square}\md \chi_p$.
\item[\textup{(d)}] If the coefficients of $f$ are integral and $k\ge 2$, then $f\md T_{p^2}\equiv f\md U_p^2\bmod p$.
\end{enumerate}
\end{lemma}

\begin{proof}[Proof of Theorem \textup{\ref{th1}}]
Consider $f\in\{f_{4a},f_{4b}\}$. For a prime $p\ge5$, the form $f$ is $p$-integral and we have $\ord_q(f)\ge-4$;
therefore Lemma~\ref{Cong1} with $k=2$ applies to result in
$$
f\md T_{p^2}^n\equiv 0 \bmod p^n.
$$
Applying Shimura--Borcherds map \eqref{SB} we deduce that, for $F=\Psi(f)\in\{F_{4a},F_{4b}\}$, we have $F\md \mathcal{T}_p^n\equiv 0 \bmod p^{n}$ for all $n\ge 1$,
hence $F\md U_{p}^n\equiv 0 \bmod p^{n}$; in other words, $F=\sum_{m>0}A(m)q^m$ has the strong $p$-magnetic property:
\begin{equation}\label{mag4}
p^n\mid m\implies p^n\mid A(m) 
\end{equation}
for any prime $p\ge 5$.
This argument also works for $f=f_{4a}$ in the case of $p=3$, because $f_{4a}$ is $3$-integral.

Consider now $p=3$ and $f=f_{4b}$, in which case we only know that $27f$ is $3$-integral.
Take the (unique!) element $g_r\in M_{5/2}^{!, +}$ with $q$-expansion $g_r=q^{-4\cdot 9^r}+O(q)$;
by \cite[Proposition~2]{DJ08} it has integral coefficients.
We first show that $g_0^\square\md T_9^n\equiv 0 \bmod 3^{n+3}$.
For $n=0$ this is true, because $g_0=-108\cdot f_{4b}$ and $f_{4b}^\square$ is in $q\Z[[q]]$.
For $n=1$ we observe that $\Psi(-\frac{1}{108}g_0\md T_9)=F_{4b}\md \mathcal{T}_3$ and $F_{4b}\equiv \Delta\bmod 3$ (since both $E_4,E_6\equiv 1\bmod 3$).
This implies that $F_{4b}\md \mathcal{T}_3\equiv\Delta\md \mathcal{T}_3\equiv 0 \bmod 3$, hence
$$
-\frac{1}{108}g_0^\square\md T_9=\Phi(F_{4b}\md \mathcal{T}_3)\equiv 0\bmod 3
$$
meaning that $g_0^\square\md T_9^n\equiv 0 \bmod 3^{n+3}$ is true when $n=1$.
Since $g_0\md T_9=27 g_1-3g_0$ we also deduce from this that $g_1^\square\equiv 0 \bmod 3$.

For $n$ general, we want to write $g_0\md T_9^n$ as a $\Z$-linear combination of $g_r$ with $r=0,1,\dots,n$.
Looking at the principal part of $g_0\md T_9^n$, one finds out that only terms of the form $q^{-4\cdot 3^{2m}}$ appear, so that subtracting the related linear combination of $f_r$ leads to a holomorphic cusp form, which then must vanish. To examine this linear combination in more details we proceed as in the proof of Lemma~\ref{Cong1}:
$$
g_0\md T_9^n
=\sum_{a,b,c,r} \alpha_{a,b,c,r}\cdot 3^{3c+b}\cdot g_0\md U_9^{a-r}\chi_3^bV_9^{c-3}
$$
(see equation~\eqref{BGK}).
As already noticed in that proof, only the terms with $r=a\le c$ contribute to the principal part, thus to the linear combination; the terms with $r=a$ contribute by the subsum
$$
\sum_{a,b,c} \alpha_{a,b,c,a}\cdot (-1)^b\cdot 3^{3c+b}\cdot g_{c-a}.
$$
Now notice that if $2c\ge a+3$, then the coefficient is divisible by $3^{n+3}$. In the remaining situations we have $2a\le 2c<a+3$, in particular $a\in\{0,1,2\}$, and we use the following analysis:
\begin{enumerate}[(a)]
\item If $a=2$, then the inequalities imply that $c=2$, hence $b=n-4$; the corresponding term is then a multiple of $3^{3\cdot 2+n-4}g_0$.
\item If $a=1$, then $c=1$, hence $b=n-2$; the corresponding term happens to be a multiple of $3^{3\cdot 1+n-2}g_0$.
\item If $a=0$, then $c\in\{0,1\}$. The term corresponding to $c=0$ is a multiple of $3^ng_0$, while the term corresponding to $c=1$ is a multiple of $3^{n+2}\cdot g_1$.
\end{enumerate}
Gathering all the terms, we end up with an expression
$$
g_0\md T_9^n=3^{n+3} g+3^{n+2}\alpha\cdot g_1+3^{n}\beta\cdot g_0,
$$
where $g$ is integral and both $\alpha$ and $\beta$ are integers.
Taking the square parts on both sides and using the results for $n=0,1$ we deduce that $g_0^\square\md T_9^n\equiv 0 \bmod 3^{n+3}$ for any $n=0,1,\dots$\,.
Finally, we apply the Shimura--Borcherds map to this congruence to deduce that $F_{4b}\md \mathcal{T}_3^n\equiv 0 \bmod 3^n$ for all $n\ge 0$. In other words, this implies the congruences \eqref{mag4} for $p=3$.

Turning now our attention to the prime $p=2$, notice that the Hecke operator $T_4$ does not respect the Kohnen plus space. However, if we define the projection
$$
K^+=K_k^+\colon \sum_{n\in\mathbb Z} a(n) q^n \mapsto \sum_{\substack{n\in\mathbb Z\\(-1)^kn\equiv0,1\bmod4}} a(n) q^n,
$$
then the operator $T_4'=K^+\circ T_4$ maps the space $M_{k+1/2}^{!,+}$ onto itself and inherits all the properties used above for $T_{p^2}$  when $p>2$.
We use this operator~$T_4'$ in place of $T_4$ to complete the proof of our Theorem~\ref{th1}.
Notice that in both cases $f=f_{4a}$ and $f=f_{4b}$ has powers of $2$ in the denominator of its main term.
For an ease of the argument, we treat the two cases separately, though the same strategy is used for both, along the line with the proof above of relation \eqref{mag4} for $p=3$.

When $f=f_{4b}$, we need to prove that $F_{4b}\md \mathcal{T}_{2}^n\equiv 0 \bmod 2^n$, which is in turn implied by the congruence  $f_{4b}^\square\md {T_4'}^n\equiv 0\bmod 2^n$. Introduce $g_r=q^{-4\cdot4^r}+O(q)\in M_{5/2}^{!,+}$ with integral $q$-expansions for $r=0,1,\dots$ and notice that $f_{4b}=-\frac{1}{108}\cdot g_0$.
The induction on $r\ge0$ shows that the recursion $g_r\md T_4'=8g_{r+1}+g_{r-1}$ takes place, with the convention that $g_{-1}=0$.
This in turn leads to
$$
g_0\md {T_4'}^n=2^{n+2}g+2^{n+1}\alpha\cdot g_1+2^n\beta\cdot g_0
$$
for some integral $g\in M_{5/2}^{!,+}$ and $\alpha,\beta\in\Z$.
Taking the square parts on both sides and using that $F_{4b}\equiv \Delta \bmod 8$, hence $\Phi(F_{4b}\md\cT_2)\equiv\Phi(\Delta\md\cT_2)\equiv0\bmod8$, we conclude with $g_0^\square\md {T_4'}^n\equiv 0\bmod 2^{n+2}$, hence with \eqref{mag4} for $p=2$ and $F=F_{4b}$.

For $f=f_{4a}$, we introduce the family $g_r=q^{-3\cdot 4^r}+O(q)\in M_{5/2}^{!,+}$, where $r=0,1,\dots$, which is invariant under the action of the operator $T_4'$, and proceed similarly to get exactly the same recursion $g_r\md T_4'=8g_{r+1}+g_{r-1}$ for $r\ge0$ with $g_{-1}=0$. On using $g_0=\frac1{64}f_{4a}$,
$$
g_0\md {T_4'}^n=2^{n+6}g+2^{n+5}\alpha\cdot g_2+2^{n+4}\beta\cdot g_1+2^{n+3}\gamma\cdot g_0
$$
for $n\ge 3$, and $F_{4a}\equiv \Delta \bmod 8$, we conclude with $g_0^\square\md {T_4'}^n\equiv 0\bmod 2^{n+6}$ implying $F_{4a}\md \mathcal{T}_{2}^n\equiv 0 \bmod 2^n$ as required.
\end{proof}

\begin{proof}[Proof of Theorem~\textup{\ref{th2}}]
We now work with $k=3$.
Consider 
$$
f(\tau)=-\frac{1}{384}\,\frac{f_{15+1/2}^*(\tau)}{\Delta(4\tau)}\in \mathcal{M}_{k+1/2}^{!,+},
$$
where $f_{b+1/2}^*$ is the weight $b+1/2$ modular form from the table in \cite[Appendix]{DJ08}.
One can easily check (through the first few coefficients) that $\Psi(f)=F_6$ and from the expression above we also know that $f$ has $p$-integral coefficients for any $p\ge 5$.
It follows from Lemma~\ref{Cong1} (applied this time with $k=3$) that $f\md T_{p^2}^n\equiv 0 \bmod p^{2n}$.
Therefore, $F_6\md \cT_p^n\equiv 0\bmod p^{2n}$ for all $n\ge 0$ implying that $F_6\md U_p^n\equiv 0 \bmod p^{2n}$ and that for $F_6=\sum_{m>0} A(m) q^m$ we have
\begin{equation}\label{cong2}
p^n\mid m\implies p^{2n}\mid A(m)
\end{equation}
for any prime $p\ge5$.

Since $384=3\cdot2^7$, for $p=3$ we see that $3f$ is $3$-integral.
Repeating the argument from Lemma~\ref{Cong1} and using the fact that $f$ is a multiple of the unique element in $ \mathcal{M}_{7/2}^{!,+}$ with the integral $q$-expansion $q^{-1}+O(q)$, we deduce that $f\md T_9^n=3^{2n}\cdot (g+ \alpha f)$ with $\alpha$ an integer and $g$ a $3$-integral modular form.
Indeed, the principal part of $f\md T_9^n$  is an $\Z$-linear combination of the principal parts of
$$
3^{(2\cdot 3-1)c+(3-1)b}\cdot f\md \chi_{3}^b V_9^{c-a}=3^{2n}\cdot (3^{c-a} f)\md \chi_{3}^b V_9^{c-a}.
$$
If $c-a\ge 1$ the principal part of $(3^{c-a} f)\md \chi_{3}^b V_9^{c-a}$ is $3$-integral; when $c=a$ the principal part of $ f\md \chi_{3}^b$ will be an integral multiple of the principal part of~$f$.
Thus, $f\md T_9^n=3^{2n}\cdot (g+ \alpha \cdot f)$ implies (applying the Shimura--Borcherds lift to both sides) that $F_6\md \cT_3^n\equiv 0 \bmod 3^{2n}$, hence we deduce that \eqref{cong2} is true also for $p=3$.

To prove the relation \eqref{cong2} for $p=2$, we proceed as in the proof of Theorem~\ref{th1}.
We introduce the $T_4'$-invariant family of weight $7/2$ weakly holomorphic modular forms $g_r=q^{-4^r}+O(q)$ with integral $q$-expansions with the help of \cite[Proposition~2]{DJ08}.
Again, we write the expression of $g_0\md {T_4'}^n$ as $\Z$-linear combination of $g_r$ with $r=0,1,\dots,n$ and analyse the powers of $2$ appearing in the coefficients; similarly, we can prove that $g_0^\square\md {T_4'}^n\equiv 0 \bmod 2^{2n+7}$ for any $n\ge 0$. For $n=0$ this comes from the integrality of $f^\square$, while for $n=1$ we get it, again, by noticing that $F_6\equiv E_6\Delta \bmod 2^{4}$ while $E_6\Delta$ being an eigenform of weight 18 with slope $4$ at the prime~$2$.
The induction argument follows \emph{mutatis mutandis} as in the proof of Theorem~\ref{th1}.
\end{proof}

\section{Miscellania on half-integral weight modular forms}
\label{w1/2}

In this part, not well related to the proofs of Theorems~\ref{th1} and~\ref{th2}, we indicate a different strategy of constructing half-integral weight weakly holomorphic modular forms using a traditional rising operator.

Standard examples of weight $1/2$ modular forms (see \cite[Sect.~14, Example 2]{Bo95}) include the theta function $\theta(\tau)=\sum_{n\in\mathbb Z}q^{n^2}$ and
\begin{align*}
h_0(\tau)
&=\frac{ E_{2,4}(\tau) \theta(\tau) \,(\theta(\tau)^4-2E_{2,4}(\tau))\,(\theta(\tau)^4-16E_{2,4}(\tau))\,E_6(4\tau)}{\Delta(4\tau)}+56\theta(\tau)
\\
&=q^{-3}-248q+26752q^4+\dotsb,
\end{align*}
where $E_{2,4}(\tau)$ is given in \eqref{E24}. The images of $12\theta$ and $4\theta+h_0$ under the multiplicative Borcherds lift
$$
\Psi^{\text{mult}}\colon\sum_{n\gg-\infty}c(n)q^n
\mapsto q^{-h}\prod_{n>0}(1-q^n)^{c(n^2)}
$$
are the modular forms $\Delta(\tau)$ and $E_4(\tau)$, respectively (see \cite[Theorem~14.1]{Bo95} for the definition of~$h$).
Although it is not useful for our results in this note, we remark that the two weakly holomorphic modular forms can serve as constructors of some weight $5/2$ modular forms from Section~\ref{w5/2}.

\begin{lemma}
\label{lem-D}
The raising operator
$$
\mathcal{D}=\mathcal{D}_k\colon f\mapsto \delta f- \frac{2k+1}{6} E_2(4\tau)\cdot f
$$
maps ${M}_{k+1/2}^{!, +}$ onto ${M}_{k+5/2}^{!, +}$.
\end{lemma}

\begin{proof}
Observe that $E_2(\tau)-4E_2(4\tau)$ is a modular form of weight 2 for $\Gamma_0(4)$, so that the difference between the usual raising operator and $\mathcal{D}$ is the multiplication by a weight 2 modular form, thus indeed $\mathcal{D}\colon{M}_{k+1/2}^{!}\to {M}_{k+5/2}^{!}$. On the other hand, both $\delta$ and multiplication by any modular form $f(4\tau)$ preserve the Kohnen plus space condition, and the lemma follows.
\end{proof}

For the functions $g_0$, $f_{4a}$ and $f_{4b}$ in Section~\ref{w5/2} we find out that
$$
g_0= -6 \mathcal{D}\theta,
\quad
64f_{4a}=-\frac{6}{19} \mathcal{D}h_0
$$
and
$$
108 f_{4b}=-\frac{6}{25}\mathcal{D}\biggl(2 h_0-1006\theta-\frac{\theta E_6(4\tau)^2}{\Delta(4\tau)}\biggr).
$$

\section{Concluding remarks}

Though we expect that our discussion above exhausts all elements with the magnetic property in $W_4^0$, many such exist for $W_{2k}^0$ with $k>2$; for example, the $q$-series $E_2^m\cdot(\delta E_j)/E_j$ for $j=4,6$ and $m=1,2,3,4,6$ (but not for $m=5$). Constructing magnetic \emph{modular} forms\,---\,meromorphic ones with poles at quadratic irrationalities from the upper half-plane\,---\,is a routine on the basis of Shimura--Borcherds (SB) lift~\eqref{SB};
Table~\ref{tab1} lists a few instances of this production explicitly in terms of the $j$-ivariant $j(\tau)=E_4^3/\Delta$.
Generating the forms with multiple magnetic property in higher weights is a tougher task;
one such example $E_4^2(j -3\cdot 2^{10})/j^2$ can be found in the more recent work \cite{LS20} of L\"obrich and Schwagenscheidt; another example of a triply magnetic form of weight~8 is
$$
E_4^2\,\frac{13 j^3 - 443556 j^2 + 1610452125 j - 98280\cdot 15^6}{(j+15^3)^4}.
$$
We have observed that in all such instances the related numerators, viewed as polynomials in~$j$, have real zeroes only.
Furthermore, there are weaker divisibility conditions (resembling the Honda--Kaneko congruences \cite{HK12}) for individual summands of magnetic forms; for example, the anti-derivatives of
$$
\frac{E_4j}{(j-2\cdot 30^3)^2} \quad\text{and}\quad \frac{E_4}{(j-2\cdot 30^3)^2}
$$
are already $p$-integral for primes $p\equiv 5\pmod6$.
We have not tried to investigate this arithmetic subphenomenon.

There is a good reason to believe that all such magnetic forms originate from suitable Shimura--Borcherds lifts.
But, maybe, there is more in this story\,---\,then time will show.

\begin{table}[h]
\caption{Strong magnetic modular forms of weight $4$ (where $f_m=q^{-m}+O(q)$ denotes the unique weakly holomorphic cusp form in $M_{5/2}^{!,+}$)}
\begin{tabular}{|c|c|}
\hline
SB lift & description in terms of $E_4$ and $j=E_4^3/\Delta$ \\
\hline
\hline
$\Psi(3^{-3} f_7)$ & $\displaystyle E_4\,\frac{19 j- 8\cdot 15^3}{(j+15^3)^2}$ \\
\hline
$\Psi(-2^{-3} f_8)$ & $\displaystyle E_4\,\frac{101 j- 3\cdot 20^3}{(j-20^3)^2}$ \\
\hline
$\Psi(2^{-6} f_{11})$ & $\displaystyle E_4\,\frac{43 j- 6\cdot 32^3}{(j+32^3)^2}$ \\
\hline
$\Psi(48^{-2}f_3\md T_4)$ & $\displaystyle E_4\,\frac{14 j+ 18\cdot 15^3}{(j-2\cdot 30^3)^2}$ \\
\hline
$\Psi(108^{-1}f_4\md (1-\tfrac{1}{2}T_4))$ & $\displaystyle E_4\,\frac{611 j+ 404\cdot 33^3}{(j- 66^3)^2}$ \\
\hline
$\Psi(3^{-3} f_7\md (2-\tfrac{1}{2}T_4))$ & $\displaystyle E_4\,\frac{82451 j+ 5272\cdot 255^3}{(j- 255^3)^2}$ \\
\hline
$\Psi(12^{-3} f_{19})$ & $\displaystyle E_4\,\frac{25 j- 2\cdot 96^3}{(j+96^3)^2}$ \\
\hline
$\Psi(12^{-3} f_{43})$ & $\displaystyle E_4\,\frac{11329 j- 578\cdot 960^3}{(j+960^3)^2}$ \\
\hline
$\Psi(12^{-3} f_{67})$ & $\displaystyle E_4\,\frac{1221961 j- 49442\cdot 5280^3}{(j+5280^3)^2}$ \\
\hline
$\Psi(12^{-3} f_{163})$ & $\displaystyle E_4\,\frac{908855380249 j- 23238932978\cdot 640320^3}{(j+640320^3)^2}$ \\
\hline
$\Psi(15^{-1} f_{15})$ & $\displaystyle E_4\,\frac{785 j^3-15219684 j^2+28709816985 j+837864\cdot 495^3}{(j^2 + 191025 j - 495^3)^2}$ \\
\hline
$\Psi(-80^{-1} f_{20})$ & $\displaystyle E_4\,\frac{733 j^3+72767680  j^2-984198615040 j+123\cdot 20^3\cdot 880^3}{(j^2-158\cdot 20^3 j-880^3)^2}$ \\
\hline
$\Psi(- f_{23})$ & $\displaystyle E_4\,\frac{P_{23}(j)}{(j^3+27934\cdot 5^3 j^2-329683\cdot 5^6 j+187^3\cdot 5^9)^2}$ \\
\hline
\end{tabular}
\hbox to\hsize{where\hfill}
\begin{align*}
P_{23}(j)&=141826 j^5-286458244\cdot 5^3 j^4+5214621227\cdot 5^6 j^3+3414887843776\cdot 5^9 j^2
\\ &\qquad
-47816219216827\cdot 5^{12} j+4378632\cdot 187^3\cdot 5^{15}
\end{align*}
\label{tab1}
\end{table}

\medskip\noindent
\textbf{Acknowledgements.}
We thank the anonymous referee for their critical reading and valuable comments:
those really helped to reduce errors and improve the exposition.


\end{document}